\documentclass{article}   	
\usepackage{geometry}               
\usepackage{graphicx}
\usepackage[usenames,dvipsnames]{xcolor}
\usepackage{amssymb,amsmath,amsthm,amsfonts}
\usepackage{bm}                     
\usepackage[english]{babel}
\usepackage[latin1]{inputenc}
\usepackage[colorlinks]{hyperref}
\hypersetup{linkcolor=blue,citecolor=blue,filecolor=black,urlcolor=blue}

\usepackage[sc]{mathpazo}

\newtheorem{proposition}{Proposition}[section]
\newtheorem{theorem}[proposition]{Theorem}

\newtheorem{lemma}[proposition]{Lemma}

\theoremstyle{definition}

\theoremstyle{remark}
\newtheorem{remark}[proposition]{Remark}
\numberwithin{equation}{section}

\newcommand{\N}{{\mathbb{N}}}

\newcommand{\R}{{\mathbb{R}}}

\newcommand{\loc}{{\text{loc}}}

\newcommand{\Dcal}{{\mathcal{D}}}

\newcommand{\adp}{{\alpha_{0}}}

\DeclareMathOperator{\dist}{dist}

\DeclareMathOperator{\divv}{div}

\DeclareMathOperator{\tr}{tr}
\DeclareMathOperator{\diam}{diam}

\newcommand{\ssubset}{\subset\joinrel\subset}

\title{Local H\"older and maximal regularity of solutions of elliptic equations with superquadratic gradient terms}

\author{Marco Cirant and Gianmaria Verzini}

\begin{document}
\maketitle

\begin{abstract} We study the local H\"older regularity of strong solutions $u$ of second-order uniformly elliptic equations having a gradient term with superquadratic growth $\gamma > 2$, and right-hand side in a Lebesgue space $L^q$. When $q > N\frac{\gamma-1}{\gamma}$ ($N$ is the dimension of the Euclidean space), we obtain the optimal H\"older continuity exponent $\alpha_q > \frac{\gamma-2}{\gamma-1}$. This allows us to prove some new results of maximal regularity type, which consist in estimating the Hessian matrix of $u$ in $L^q$. Our methods are based on blow-up techniques and a Liouville theorem.
\end{abstract}
\noindent
{\footnotesize \textbf{AMS-Subject Classification}}. 
{\footnotesize 35B65, 35J60, 35R05}

\section{Introduction}

The goal of this paper is to address some regularity issues related to the elliptic PDE
\begin{equation}\label{eq:HJ_main}
-\tr \left(A(x) D^2 u\right) + H(x,Du) = f(x)\qquad\text{in }\Omega \subset \R^N,
\end{equation}
where $A(x)$ is a nondegenerate diffusion matrix, $H$ has superquadratic growth in the gradient variable, and $f$ belongs to some Lebesgue space $L^q(\Omega)$. We are interested in the H\"older regularity of the solution $u$, as well as the regularity of the gradient $D u$ and the Hessian matrix $D^2 u$ in $L^q$.

Equations of the form \eqref{eq:HJ_main} appear naturally in the theory of (ergodic) stochastic control, homogenization, in the theory of growth of surfaces, and in differential games with many players. They appear in the literature under different names, such as (viscous) Hamilton-Jacobi equations, KPZ or Riccati equations. As a prototype of semilinear equation with superlinear character in the first order entry, the regularity properties of \eqref{eq:HJ_main} have been extensively investigated. Nevertheless, some recent questions concerning the ``maximal regularity'' of solutions demand for a deeper understanding of the interaction between the linear second order diffusion and the nonlinear first order term.

Let us start with some considerations on the H\"older regularity of solutions to \eqref{eq:HJ_main}, having in mind the model superquadratic Hamiltonian
\begin{equation}\label{modelH}
H(x, Du) = |Du|^\gamma, \qquad \gamma > 2.
\end{equation}
A few years ago, A. Dall'Aglio and A. Porretta showed in \cite{Dall_Aglio_2015} that \textit{weak} solutions are $\frac{\gamma-2}{\gamma-1}$-H\"older continuous, provided that $f \in L^q$ and
\begin{equation}\label{q0def}
q = q_0 := \frac{N}{\gamma'}=N \frac{\gamma-1}{\gamma}.
\end{equation}
Their result is for equations in divergence form, and it is genuinely perturbative: the diffusion term plays no role (it can be even degenerate), and the regularity is a sole byproduct of the coercivity of $H$. In some sense, it cannot be even improved, since 
\begin{equation}\label{badu}
\text{$u(x) = c|x|^{\frac{\gamma-2}{\gamma-1}}$ \ is a weak solution of \ $-\Delta u + |Du|^\gamma = 0$ on $\R^N$}
\end{equation}
for suitable $c \in \mathbb R$. Note that $\frac{\gamma-2}{\gamma-1}$-H\"older estimates are true even for \textit{subsolutions}, and hold up to the boundary: this shows that superquadratic problems enjoy some properties that are rather unnatural for elliptic equations. Besides, solutions to uniformly elliptic (quasilinear) equations having subquadratic (subnatural) growth in the gradient are known to be H\"older continuous when $q > \frac N 2$ (see e.g. \cite{Bensoussan_2002, LU}), and this classical fact leans on the perturbative nature of $H(Du)$ when $\gamma < 2$.

If one looks at solutions that are more than just weak, better a priori estimates can be obtained. In \cite{Lions_1985}, P.-L. Lions showed that Lipschitz estimates for classical solutions can be achieved, provided that $q > N$, in the full superlinear range $\gamma > 1$. It is worth remarking that these estimates were obtained via the Bernstein method, which allowed to exploit both the regularizing effects of the diffusion and the coercivity of $H$ (and therefore, by means of a nonperturbative argument). Lipschitz estimates have been obtained later in \cite{Capuzzo_2010} for viscosity solutions, and for equations with a possibly degenerate diffusion matrix, assuming in addition $f$ to be Lipschitz continuous.

Our first goal is to fill the gap in the understanding of $\alpha$-H\"older regularity of $u$, in the range
\[
\frac{\gamma-2}{\gamma-1} < \alpha < 1 \qquad\text{with } f \in L^q, \quad \frac N {\gamma'} < q < N, 
\]
for solutions in the \textit{strong} sense, which is naturally intermediate between weak and classical. To achieve this goal, we will develop a nonperturbative method, which will again exploit first and second order regularizing effects. These new H\"older regularity results will imply in a straightforward way new results regarding the so-called $L^q$-maximal regularity.

\smallskip

The problem of $L^q$-maximal regularity for \eqref{eq:HJ_main} has been raised by P.-L. Lions a decade ago in a series of seminars, and, roughly speaking, is a semilinear version of the classical Calder\'on-Zygmund linear maximal regularity; that is, under the assumption that $f$ is bounded in $L^q$, then one should be able to have a control of $H(Du)$ and $ \tr \left(A D^2 u\right)$ in $L^q$. Lions conjectured that this should be possible provided that $q > q_0$ ($q_0$ as in \eqref{q0def}), and the conjecture has been shown to be true in the recent work \cite{Cirant_2021}. The results in \cite{Cirant_2021} still have some limitations: they hold for classical solutions, $q > 2$ is required (while it may happen that $q_0 < 2$ for some $N, \gamma$) and they are \textit{not} local but require $f(x)$ to be periodic in $x$. Moreover, no $x$ dependence is allowed in first and second order terms. 

Our second goal here is to circumvent all these limitations. The results in \cite{Cirant_2021} are again based on a Bernstein method. Apart from the issues on the generality of the statements in \cite{Cirant_2021}, a crux is that the Bernstein approach may not be employed for different problems, e.g. involving degenerate or fractional diffusions. Therefore, a general goal of this work is to develop a different approach to maximal regularity, which can be flexible enough to be applied to a wider range of equations (which will be briefly described at the end of this introduction). The core idea of the method proposed here is to get maximal regularity using H\"older estimates as an intermediate step. Nevertheless, (linear) maximal regularity will be crucial to obtain such H\"older estimates, so one may argue, after reading the proofs, that regularity at these two scales is in fact interconnected.

\smallskip

We now discuss our standing assumptions. $\Omega$ is a bounded Lipschitz domain, with uniform interior sphere property. Concerning $A$, we assume that constants $0<\lambda<\Lambda$ exist such that
\begin{equation}
\tag{Ass$_A$}\label{eq:assA} A \in C(\overline{\Omega}), \qquad \divv A \in L^N(\Omega) ,\qquad \lambda|\xi|^2 \le  A\xi\cdot\xi\le \Lambda|\xi|^2 \text{ on } \overline{\Omega}.
\end{equation}
Moreover we assume, that for some $\gamma>2$, 
\begin{equation}
\tag{Ass$_H$}\label{eq:assH} 
\begin{split}
&H(x,p) = h(x) |p|^\gamma + H_{0}(x,p),\qquad \text{where $h\ge h_{\min}>0$ is continuous in $\overline{\Omega}$}\\
&\text{and }|H_0(x,p)|\le C_1|p|^{\gamma_1}+C_2,\qquad
\text{with }0\le\gamma_1<\gamma,\ C_1,C_2\ge0.
\end{split}
\end{equation}
For brevity, let
\[
\adp = \frac{\gamma-2}{\gamma-1}, \qquad\text{corresponding to}\qquad q_0=\frac{N}{\gamma'}
\]
As we mentioned, the $\alpha_0$-H\"older regularity of $u$, which requires $f \in L^q(\Omega)$, $q = q_0$ will be our starting point. In the following remark, which will be useful in the sequel, we comment how the regularity of $u$ at different scales is given by the embeddings.
\begin{remark}\label{rem:rightalfa} Assume that
\[
\frac{N}{\gamma'} \le q < N.
\] 
If $u\in W^{2,q}(\Omega)$ then, by Sobolev embedding, we have that
$Du\in L^{q^*}(\Omega)$ and $u \in C^{0,\alpha}(\overline{\Omega})$, where
\[
q^* = \frac{Nq}{N-q} \ge \gamma q \qquad\text{and}\qquad
\alpha = 2 - \frac{N}{q} \ge \adp,
\]
and equalities hold if and only if $q = q_0 = \frac{N}{\gamma'}$.
\end{remark}

\smallskip

The main results on the H\"older regularity reads as follows.
\begin{theorem}\label{thm:main_intro}
Let $q>\frac{N}{\gamma'}$. Assume that $\alpha = 2 - \frac{N}{q}$ if $q<N$, or 
$\alpha<1$ if $q\ge N$. For every $M\ge0$ there exists $C=C(M, N, q, \alpha, H, A, \Omega)$ such that if 
$u\in W^{2,q}(\Omega)$ solves 
\eqref{eq:HJ_main} in $\Omega$ in the strong sense, with $\|f\|_q\le M$, then
\begin{equation}\label{mainHolderest}
\sup_{\bar x \neq x} \min\{\dist(\bar x,\partial\Omega),\dist(x,\partial\Omega)\}^{\alpha-\alpha_0}
\,\frac{|u(\bar x)-u(x)|}{|\bar x-x|^\alpha} \le C .
\end{equation}
\end{theorem}

Note that the $\alpha$-H\"older seminorm locally deteriorates as $x, \bar x$ approach $\partial 
\Omega$. This has to be expected, as the function $u$ in \eqref{badu} is a classical solution in 
any domain $\Omega$ such that $0 \in \partial \Omega$, and it is not better than $\alpha_0$-H\"older on $\Omega$. In fact, the weight in \eqref{mainHolderest} is sharp. Note also that, as $
\alpha \to 1$, such weight agrees with the one appearing in Lipschitz estimates which were obtained in 
\cite{Capuzzo_2010} for viscosity solutions. Finally, the constant $C$ depends actually on $
\lambda, \Lambda$ in \eqref{eq:assA}, $\|\divv A\|_{L^N(\Omega)}$, $C_1, C_2, \gamma_1, \gamma, 
h_{\min}$ in \eqref{eq:assH}, the moduli of continuity of $A$ and $h_0$ on $\overline{\Omega}$ and $\Omega$ itself.

The proof of Theorem \ref{thm:main_intro} relies on a blow-up argument. We employ a Liouville theorem for the homogeneous version of \eqref{eq:HJ_main} on $\R^N$. The compactness which is necessary to pass to the limit in the scaling procedure involves an interpolation argument.

\smallskip

As a consequence of Theorem \ref{thm:main_intro}, we get the following $L^q$-maximal regularity result.
\begin{theorem}\label{thm:main_intro2}
Let $q>\frac{N}{\gamma'}$. For every $M\ge0$ and $\Omega' \ssubset \Omega$ there exists $C=C(M$, $\dist(\Omega', \partial \Omega)$, $N, q$, $\alpha, H, A, \Omega)$ such that if $u\in W^{2,q}(\Omega)$ solves
\eqref{eq:HJ_main} in $\Omega$ in the strong sense, with $\|f\|_q\le M$, then
\[
\|u\|_{W^{2,q}(\Omega')} \le C.
\]
\end{theorem}

Both theorems hold for merely $W^{2,q}_{\rm loc}(\Omega)$ solutions, as one can apply them in exhaustions of $\Omega$ (and constants depend only on the diameter of the domain). Moreover, one could quantify more precisely how the $\|u\|_{W^{2,q}(\Omega')}$ norm deteriorates as $\dist(\Omega', \partial \Omega) \to 0$, but we avoid this computation for brevity.

\smallskip

We now discuss further possible generalizations. First, we expect that $ \divv A \in L^N(\Omega)$ could be avoided. In fact, this assumption is used only to trigger the $\alpha_0$-H\"older estimates, which were obtained for equations in divergence form. Indeed, $ \divv A \in L^N(\Omega)$ is needed when transforming \eqref{eq:HJ_main} into a divergence form equation. We believe that $\alpha_0$-H\"older estimates could be obtained directly for \eqref{eq:HJ_main}, and for merely continuous $A$. Note that this question is related with the $L^q$-maximal regularity in the critical case $q = q_0$. Despite a statement as in Theorem \ref{thm:main_intro2} cannot hold (again by \eqref{badu} it is possible to construct a sequence of $u_n, f_n$ solving \eqref{eq:HJ_main} such that $u_n \to \infty$ in $W^{2,q_0}$ while $f_n$ remains bounded in $L^{q_0}$, see e.g. \cite{Cirant_2021}), it should be possible to control $u$ in $W^{2,q_0}$ whenever $f$ lies in a set of $L^{q_0}$ uniformly integrable functions. This has been shown to hold for subquadratic ($\gamma < 2$) problems \cite{Goffi_arxiv}.

We also hope that our strategy can be extended to different settings. Let us mention the cases where $\tr \left(A(x) D^2 u\right)$ is replaced by a nonlocal diffusion operator, or $Du, D^2 u$ become the horizontal gradient and horizontal Hessian respectively induced by a family of H\"ormander vector fields, as in \cite{Bardi_1991}. We stress again that, to achieve full $L^q$-maximal regularity, the adaption of the methods in \cite{Cirant_2021} are by no means obvious. On the other hand, the strategy presented here basically relies on Liouville theorems, which can be obtained in rather different ways. We also mention that, for parabolic problems, an exhaustive picture on the maximal regularity of solutions is still missing, despite some partial results appeared in \cite{CG_2021}.

Another direction would be towards quasilinear equations, modeled for example on the $p$-Laplacian operator. The extensive literature on the gradient regularity cannot be summarized here, see for example the recent survey \cite{Palatucci_2020} and \cite{Cianchi_2014}. The case with strong first order terms (in particular beyond the natural growth) appears to be still an open research field \cite{Leonori_2016, Phuc_2020}. Liouville theorems are known \cite{Bidaut-Veron_2014}, hence we believe that the techniques presented here could yield new results in this direction.

Finally, a challenging goal would be to address the H\"older regularity for systems of HJ equations. Indeed, for such kind of systems, it is known that the equivalence between Liouville results and interior H\"older regularity in general holds only in some restricted sense \cite{Meier}.

\bigskip

\textbf{Acknowledgements.} The authors are members of the Gruppo Nazionale per l'Analisi Matematica, la Probabilit\`a e le loro Applicazioni (GNAMPA) of the Istituto Nazionale di Alta Matematica (INdAM). Work partially supported by the project Vain-Hopes within the program 
VALERE-Università degli Studi della Campania ``Luigi Vanvitelli'', and by the Portuguese 
government through FCT/Portugal under the project PTDC/MAT-PUR/1788/2020.

\bigskip

\textbf{Notations.}
\begin{itemize}
\item $\|v\|_{r;D} = \|v\|_{L^r(D)}$; \quad $[v]_{\alpha;D} = [v]_{C^{0,\alpha}(\overline{D})}$;
\item $\adp = \frac{\gamma-2}{\gamma-1}$; \quad $q_0 = N \frac{\gamma-1}{\gamma}$; \quad $\gamma'=\frac{\gamma}{\gamma-1}$; \quad  $q^* = \frac{Nq}{N-q}$;
\item $C,C'$ and so on denote non-negative universal constants, which we need 
not to specify, and which may vary from line to line.
\end{itemize}

\section{Preliminaries}

Our starting point is the $\alpha_0$-H\"older estimates. This is a version of the result by Dall'Aglio and Porretta, simplified for our 
purposes. 
\begin{lemma}[{\cite[Thm. 1.1]{Dall_Aglio_2015}}]\label{lem:DAP} 
Let $\Omega\subset\R^N$ be a bounded Lipschitz domain, satisfying the uniform interior sphere condition. 
Assume that $\gamma>2$, $\tilde f\in L^q(\Omega)$, for some 
$q\ge \frac{N}{\gamma'}$ and that 
$u\in W^{1,\gamma}_\loc(\Omega)$ satisfies in a distributional sense 
\[
-\divv (a (x,u,D u)) + |Du|^\gamma \le \tilde f\qquad \text{in }\Omega,
\] 
where $|a(x,s,\xi)|\le\beta(1+|\xi|)$ for some $\beta>0$. Then $u$ is  
$\alpha_0$-H\"older continuous up to the boundary of $\Omega$, and there exists  a constant $K$, only depending 
on $\gamma$, $q$, $N$, $\beta$ $\diam(\Omega)$ and $\|f\|_q$, such that
\[
|u(x)-u(\bar x)|\le K|x-\bar x|^\adp,\qquad\text{for every }x,\bar x\in \overline{\Omega}.
\]
\end{lemma}
%
%
\begin{remark}\label{rem:divnondiv}
Since
\[
-\divv \left(A(x) \nabla u\right) = -\tr \left(A(x) D^2 u\right) - b(x)\cdot \nabla u,
\]
where $b = (\divv A)^T = (\sum_i \partial_i a_{i1},\dots,\sum_i \partial_i 
a_{iN})$, we infer that if $u$ satisfies \eqref{eq:HJ_main} then it fulfills the assumptions of Lemma \ref{lem:DAP}, as long as $A$ satisfies \eqref{eq:assA} and $H$ is such that
\[
H(x,p) \ge C |p|^\gamma -  C'\qquad\text{in }\Omega,
\]
for some positive constants $C,C'$ (in particular this is true under 
\eqref{eq:assH}, because $\min_{\overline{\Omega}}h>0$). Indeed, we have, by Young's inequality,
\begin{align*}
-\divv \left(A(x) \nabla u\right) + C |\nabla u|^\gamma & \le f  + C' - b \cdot \nabla u \\
& \le f + C' + \frac C 2 |\nabla u|^\gamma + C'' |b|^{\gamma'},
\end{align*}
and Lemma \ref{lem:DAP} applies with $\tilde f = f + C' + C'' |\divv A|^{\gamma'}$.
\end{remark}

Below we state the local version of the classical linear Calder\`on-Zygmund elliptic regularity result.

\begin{lemma}[{\cite[Thm. 9.11, eq. (9.40)]{Gilbarg_2001}}]\label{lem:GT}
Let $\Omega$ be an open set in $\R^N$ and $u\in W^{2,q}_\loc(\Omega)\cap L^q(\Omega)$, $1<q<\infty$, be a strong solution of 
\[
-\tr \left(A(x) D^2 u\right) = g\qquad\text{in }\Omega,
\]
where $A$ satisfies \eqref{eq:assA} and $g\in L^q(\Omega)$. There exists a constant 
$C$, only depending on $N$, $q$, $\lambda$ and $\Lambda$, and $\delta$ depending also on the modulus of continuity of $A$ such that, for every ball $B_R\subset\subset\Omega$ with $R \le \delta$ and for every $0<\sigma<1$, 
\[
\|D^2u\|_{q;B_{\sigma R}}\le \frac{C}{(1-\sigma)^2R^2}\left(
R^2\|g\|_{q;B_{R}}+ \|u\|_{q;B_{R}}
\right).
\]
\end{lemma}

We also need the following particular case of the Gagliardo-Nirenberg inequality.
\begin{lemma}[{\cite{Nirenberg_1966}}]\label{lem:GN}
Let $\gamma>2$, $\frac{N}{\gamma'}<q<N$, $\alpha=2-\frac{N}{q}$. There exist 
$C_1,C_2$ independent of $R$ such that 
\[
\|Dw\|_{q^*;B_R} \le C_1\|D^2w\|_{q;B_R}^a[w]_{\alpha;B_R}^{1-a} + C_2 [w]_{\alpha;B_R},
\qquad \text{where }a=1-\frac{q}{N}<\frac{1}{\gamma}\quad ,
\]
for every $w\in W_{\loc}^{2,q}(\Omega)$, and $B_R \ssubset \Omega$.
\end{lemma}
\begin{proof}
By Remark \ref{rem:rightalfa} if $w\in W^{2,q}(\Omega)$ all the above 
quantities are finite. When $R=1$ the existence of $C_1,C_2$ is obtained in 
\cite[Thm. $1'$]{Nirenberg_1966}, and in particular assumption $(4)'$ therein 
is fulfilled as  
\[
\frac{1-\alpha}{2-\alpha} \le a
\]
By scaling, it is easy to check that the inequality holds in $B_R$ with the 
same constants.
\end{proof}

Finally, we need the following Liouville-type result.
\begin{lemma}\label{lem:Lions} 
Let $A_0$ be a constant, symmetric and positive definite matrix, $\gamma>2$, $h_0$ be a 
positive constant, and  $w\in W^{2,q}_\loc(\R^N)$, $q>\frac{N}{\gamma'}$, solve
\[
-\tr \left(A_0 D^2 w\right) + h_0|Dw|^\gamma = 0\qquad \text{in }\R^N.
\]
Then $w$ is constant.
\end{lemma}
This result, in case $w$ is $C^2$ and $A=\mathrm{Id}$, is 
\cite[Cor. IV.2]{Lions_1985} (see also the previous work \cite{Peletier_Serrin}). As claimed there, it can be extended to more general 
situations, like the one we need here. For the reader's convenience we provide a 
proof (which could be generalized also to the range $1 < \gamma \le 2$). 
\begin{proof}
The proof is based on a Bernstein-type argument similar to that used in  
Theorem IV.1 in \cite{Lions_1985}, with some
standard modification to deal with non-divergence equations, see e.g. 
\cite{Evans_1985} or \cite{Bardi_1991}. First observe that, by a standard 
bootstrap argument (triggered by Remark \ref{rem:rightalfa}), if $w\in 
W^{2,q}_\loc(\R^N)$ solves the above equation then $w$ is actually $C^3$.

For a cut-off function $\varphi\in\Dcal(B_{1})$, $0\le \varphi\le 1$, $\varphi(0)=1$, define
\[
z(x) = \xi^2(x)|Dw(x)|^2,\qquad \text{where }\xi(x) = \varphi\left(\frac{x-x_0}{R}\right)
\]
where $R\ge1$ and $x_0\in\R^N$. Moreover
\[
I=-\tr(A_0 D^2 z) + h_0\gamma|Dw|^{\gamma-2}Dw\cdot Dz.
\] 
Since both $A_0=(a_{ij})$ and $D^2w$ are symmetric, direct calculations yield (under the convention of repeated indexes summation)
\[
I= -a_{ij}(\xi^2)_{ij} |Dw|^2 - 4 a_{ij} (\xi^2)_i w_l w_{lj} - 2\xi^2a_{ij}w_{li}w_{lj} + 
h_0\gamma|Dw|^{\gamma}Dw\cdot D(\xi^2).
\]
Now, for suitable constants $C>0$, and $0< \lambda I \le A_0 \le \Lambda I$,
\begin{itemize}
\item $|D \xi |\le \frac{C}{R}\xi^{1/2}$, $|D^2 \xi |\le \frac{C}{R^2}$;
\item $a_{ij}(\xi^2)_{ij} = a_{ij}\xi\xi_{ij} + a_{ij}\xi_i\xi_{j} \ge -\frac{C}{R^2} \xi$;
\item $\left|a_{ij} (\xi^2)_i w_l w_{lj}\right|\le \frac{C}{R}\xi^{3/2}|Dw||D^2w| \le
\frac{C'}{R^2}\xi|Dw|^2 + \lambda\xi^2|D^2w|^2 $;
\item $2\xi^2 a_{ij}w_{li}w_{lj}\ge 2\lambda\xi^2 |D^2 w|^2 $;
\item $\lambda |D^2 w|^2 \ge 
\frac{\lambda}{N\Lambda^2} \tr(A_0^2) \tr((D^2w)^2) \ge \frac{\lambda}{N\Lambda^2} \tr(A_0 D^2w)^2 = \frac{\lambda h_0^2}{N\Lambda^2} |Dw|^{2\gamma}$. 
\end{itemize}
We obtain
\[
\begin{split}
I &\le C\left[\frac{1}{R^2}\xi|Dw|^2+\frac{1}{R}\xi|Dw|^{\gamma+1}-2\xi^2|Dw|^{2\gamma}\right]\\
&= C\left[\xi|Dw|^2\left(\frac{1}{R^2}-\xi|Dw|^{2(\gamma-1)}\right)+\xi|Dw|^{\gamma+1}\left(\frac{1}{R}-\xi|Dw|^{\gamma-1}\right)\right]
\\
&\le C\left[\xi|Dw|^2\left(\frac{1}{R^2}-z^{\gamma-1}\right)+\xi|Dw|^{\gamma+1}\left(\frac{1}{R}-z^{(\gamma-1)/2}\right)\right].
\end{split}
\]
Finally, let $x_M$ be a maximum point for $z$. Then $I(x_M)\ge0$ and the previous inequality 
yields
\[
|D w(x_0)|^2 = z(x_0) \le z(x_M)\le \frac{1}{R^{2(\gamma-1)}}.
\]
The lemma follows since $R\ge1$ and $x_0\in\R^N$ are arbitrary. 
\end{proof}

\section{Proof of Theorems \ref{thm:main_intro} and \ref{thm:main_intro2}}

To prove Theorem \ref{thm:main_intro} we develop a blow-up procedure based on 
a contradiction argument. Throughout this section we fix
\[
q\in\left(\frac{N}{\gamma'}, N\right) \qquad\text{and}\qquad
\alpha = 2 - \frac{N}{q} \in\left( \adp , 1\right).
\]
We assume by contradiction the existence of 
sequences $(f_n)_n\subset L^q(\Omega)$, $(u_n)_n\subset W^{2,q}(\Omega)$ 
satisfying, for every $n$:
\begin{enumerate}
\item $-\tr \left(A(x) D^2 u_n\right) + H(x,Du_n) = f_n(x)$;
\item $\|f_n\|_q \le M$;
\item $\displaystyle\sup_{\bar x \neq x} \min\{\dist(\bar x,\partial\Omega),\dist(x,\partial\Omega)\}^{\alpha-\alpha_0}\,\frac{|u_n(\bar x)-u_n(x)|}{|\bar x-x|^\alpha} =: L_n \to +\infty$ as $n\to+\infty$.
\end{enumerate}
Notice that, since $W^{2,q}(\Omega)\hookrightarrow C^{0,\alpha}(\overline{\Omega})$ 
(see Remark \ref{rem:rightalfa}), $L_n$ is finite for every $n$. Consequently, 
there exist sequences $(x_n)_n$ and $(\bar x_n)_n$ in $\Omega$ such that, for every 
$n$,  $x_n\neq \bar x_n$, $\dist(\bar x_n,\partial \Omega)\le \dist(x_n,\partial \Omega)$ and
\begin{equation*}
L_n-1\le \dist(\bar x_n,\partial\Omega)^{\alpha-\alpha_0}
\,\frac{|u_n(\bar x_n)-u_n(x_n)|}{|\bar x_n-x_n|^\alpha} \le L_n.
\end{equation*}
Writing 
\[
r_n = |\bar x_n - x_n|,\qquad d_n = \dist(\bar x_n,\partial\Omega) \qquad
\text{and, w.l.o.g., } u_n(\bar x_n) = 0
\]
we obtain  
\begin{equation}\label{eq:seq2}
L_n-1\le d_n^{\alpha-\alpha_0}
\,\frac{|u_n(x_n)|}{r_n^\alpha} \le L_n
\end{equation}
(in particular, $u_n(x_n)\neq0$ for every $n$).
\begin{lemma}\label{lem:consDP}
There exists $K$, only depending on the structural constants and $M$, such that, 
for every $n$, 
\[
\frac{|u_n(x_n)|}{r_n^\adp}\le K.
\]
Consequently, 
\[
\frac{d_n}{r_n}\to+\infty,\qquad r_n\to0,\qquad |u_n(x_n)|\to0 
\qquad\text{and }\quad\frac{r_n^{\beta}}{|u_n(x_n)|}\to0,\ \beta>\alpha_0,
\]
as $n\to+\infty$.
\end{lemma}
\begin{proof}
Recalling that $u(\bar x_n)=0$ the first assertion follows by Lemma \ref{lem:DAP}, see also Remark 
\ref{rem:divnondiv}. From \eqref{eq:seq2} we infer
\[
\left(\frac{d_n}{r_n}\right)^{\alpha-\adp} \ge  (L_n-1) \frac{r_n^\adp}{|u_n(x_n)|}
\ge  \frac{L_n-1}{K}
\]
and the lemma follows.
\end{proof}

We introduce the  blow-up sequence
\[
w_n(y):=\frac{1}{|u_n(x_n)|} u_n(\bar x_n + r_n y), \qquad y\in \Omega_n:=
\frac{\Omega-\bar x_n}{r_n}.
\]
Of course, $w_n\in W^{2,q}(\Omega_n)$, for every $n$. We first show that, 
by contruction, $\Omega_n$ invades the whole space, and that $(w_n)_n$ is 
locally equi-H\"older with exponent $\alpha$.
\begin{lemma}\label{lem:invasion}
Let $R>0$ be fixed and $B_R=B_R(0)$. Then 
\[
\Omega_n\supset B_R\qquad\text{and}\qquad [w_n]_{\alpha;B_R}\le 2
\]
for $n$ sufficiently large. In particular, $w_n\to w_\infty$ uniformly on compact sets, up to subsequences, and $w_\infty$ is not constant.
\end{lemma}
\begin{proof}
By Lemma \ref{lem:consDP} we have
\[
\dist(0,\partial\Omega_n) = \max_{y\in\partial \Omega_n} |y| = \max_{x\in\partial \Omega} \frac{|x-\bar x_n|}{r_n} = \frac{d_n}{r_n}\ge R
\]
for $n$ large, and the first assertion follows.

Now let $y,y'\in B_R \subset \Omega_n$, and assume that, up to subsequences, 
$\dist(y',\partial\Omega_n)\le \dist(y,\partial\Omega_n)$. Then, writing 
$x=\bar x_n + r_n y$, $x'=\bar x_n + r_n y'$ we infer that $x'_n\in
B_{Rr_n}(\bar x_n)$, so that
\[
\dist(x,\partial\Omega)\ge \dist(x',\partial\Omega) \ge 
\dist(\bar x_n,\partial\Omega) - Rr_n=d_n - Rr_n.
\]
Thus
\[
\frac{| u_n(x)- u_n(x')|}{|x-x'|^\alpha} \le \frac{L_n}{\dist(x',\partial\Omega_n)^{\alpha-\adp}}\le \frac{L_n}{(d_n - Rr_n)^{\alpha-\adp}},
\]
whence
\[
\frac{|w_n(y)-w_n(y')|}{|y-y'|^\alpha} = \frac{| u_n(x)- u_n(x')|}{|x-x'|^\alpha}
\,\frac{r_n^\alpha}{|u_n(x_n)|}\le \frac{L_n}{(d_n - Rr_n)^{\alpha-\adp}} \,
\frac{d_n^{\alpha-\adp}}{L_n-1}
\]
by \eqref{eq:seq2}. Since $L_n\to+\infty$ and (by Lemma \ref{lem:consDP}) $r_n=o(d_n)$ as $n\to+\infty$, also the second assertion follows.
\end{proof}

By direct computation it is easy to check that $w_n$ solves
\begin{equation}\label{eq:bu_sqn}
-\tr\left(A_n(y) D^2w_n\right) + \frac{r_n^2}{|u_n(x_n)|} 
H_n\left(y, \frac{|u_n(x_n)|}{r_n}Dw_n\right) = g_n(y)\qquad \text{in }\Omega_n,
\end{equation}
where
\[
A_n(y) = A(\bar x_n + r_n y), \quad H_n(y,p) = H(\bar x_n + r_n y,p), \quad
g_n(y) =  \frac{r_n^2}{|u_n(x_n)|} f_n(\bar x_n + r_n y).
\]
In order to pass to the limit in \eqref{eq:bu_sqn} we are going to show 
uniform boundedness in $W^{2,q}_\loc$, by means of an iterative argument. The desired regularity is achieved via elliptic estimates and interpolation. The following argument is inspired by ideas in \cite{Amann_1978}.

\begin{proposition}\label{prop:holdertoW2q} Let $R > 0 $, $g \in L^q(B_{2R})$ and $v \in W^{2,q}(B_{2R})$ be such that $v(0) = 0$,
\[
\|g\|_{q, B_{2R}} + [v]_{\alpha, B_{2R}} \le c_1\qquad \left| \tr\left(A(x) D^2 v(x) \right)\right| \le c_2|D v(x)|^\gamma + g(x)
\]
a.e. in $B_{2R}$, for some $c_1, c_2 > 0$. Then, there exists $K$ depending on $c_1, c_2, R$ such that
\[
\|D^2 v\|_{q,B_{R}}\le K.
\]
\end{proposition}
\begin{proof} First, note that since $[v]_{\alpha, B_{2R}} \le c_1$ and $v(0) = 0$,
\begin{equation}\label{eq:boundwn1}
\|v\|_{q;B_{2R}}\le c_1 C  R^{\alpha+\frac{N}{q}} = c_1 C R^2.
\end{equation}
Let $R \le \rho \le 2R$ and $0 < \sigma < 1$. Assume first that $2R \le \delta$, $\delta$ be as in Lemma \ref{lem:GT}. Applying such lemma, H\"older's inequality, and using \eqref{eq:boundwn1} we obtain
\begin{equation*}
\begin{split}
\|D^2 v\|_{q;B_{\sigma \rho}}&\le \frac{C}{(1-\sigma)^2R^2}\left(
R^2 \|  c_2|D v|^\gamma + g \|_{q;B_{\rho}} + \|  v \|_{q;B_{\rho}}
\right)\\
&\le \frac{C}{(1-\sigma)^2}\left(
R^{\gamma-\frac{N}{q}(\gamma-1)}\|D v\|_{q^*;B_\rho}^\gamma+1
\right).
\end{split}
\end{equation*}
Therefore, applying Lemma \ref{lem:GN} we have
\begin{equation}\label{eq:smallvsbig}
\|D^2 v\|_{q;B_{\sigma \rho}} \le  \frac{C}{(1-\sigma)^2}\left(R^{\gamma-\frac{N}{q}(\gamma-1)}
\|D^2 v\|^{a\gamma}_{q;B_\rho}+R^{\gamma-\frac{N}{q}(\gamma-1)}+1
\right).
\end{equation}
\smallskip

Now, if $\|D^2 v\|^{a\gamma}_{q,B_{R}}\le 1 + R^{-\gamma+\frac{N}{q}(\gamma-1)}$ then there is nothing 
to prove. Otherwise, \eqref{eq:smallvsbig} yields, for every $R\le \rho \le
2R$ and $0<\sigma<1$,
\begin{equation}\label{eq:iter_1}
\|D^2 v\|_{q,B_{\sigma \rho}}\le\frac{ E}{(1-\sigma)^2}
\|D^2 v\|_{q;B_\rho}^{a\gamma} \ ,
\end{equation}
where $E =CR^{\gamma-\frac{N}{q}(\gamma-1)}$, and $C$ is independent of $\rho,\sigma$. For
$k\in\N$ we write
\[
\rho_k=(2-2^{-k})R,\quad 1-\sigma_{k+1} = \frac{\rho_{k+1}-\rho_k}{\rho_{k+1}} \ge 2^{-(k+2)},\quad \phi_k = \log_2 \|D^2 v\|_{q;B_{\rho_k}} > 0.
\]
Then \eqref{eq:iter_1} implies, for every $k$,
\[
\phi_{k} \le \log_2 E + 2(k+2) + a\gamma\phi_{k+1}
\]
and, by induction,
\[
\phi_{0} \le (4 + \log_2 E)\sum_{i=0}^{k-1} (a\gamma)^i + 2\sum_{i=0}^{k-1} i (a\gamma)^i + (a\gamma)^{k}\phi_{k},
\]
for every $k$. Let $k\to+\infty$. Since $a\gamma<1$ and $\phi_k\le 
\log_2 \|D^2 v\|_{q;B_{2R}}<+\infty $ we obtain 
\[
\phi_{0} \le \frac{4 + \log_2 E}{1-a\gamma} + \frac{2}{(1-a\gamma)^2},
\]
and the conclusion follows by the definitions of $\phi_0$ and $E$.

The case $2R > \delta$ follows by a standard covering argument.
\end{proof}
We now apply the previous proposition to the blow-up sequence $w_n$, which solves  \eqref{eq:bu_sqn}.
\begin{lemma}\label{lem:iter}
For every $R\ge 1$ there exists a constant $C_R$ such that
\[
\|D^2 w_n\|_{q,B_{R}}\le C_R,
\]
for $n$ sufficiently large.
\end{lemma}
%
%
\begin{proof}
Let $R\ge1$ be fixed. By Lemma \ref{lem:invasion}, if $n$ is sufficiently large 
then $B_{2R}\subset\Omega_n$, whence $w_n\in W^{2,q}(B_{2R})$. We want to apply Proposition \ref{prop:holdertoW2q}, as
\begin{equation*}
\left|\tr\left(A_n(y) D^2w_n\right)\right| \le \frac{r_n^2}{|u_n(x_n)|}  \left| 
H_n\left(y, \frac{|u_n(x_n)|}{r_n}Dw_n\right) \right| + |g_n(y)|\qquad \text{in }B_{2R}.
\end{equation*}
To this aim we notice that, on the one hand, by \eqref{eq:assH} we know that $|H(x,p)|\le C_1|p|^\gamma + C_2$, for some 
$C_1,C_2$. Then Lemma \ref{lem:consDP} yields
\begin{equation*}
\begin{split}
\left|\frac{r_n^2}{|u_n(x_n)|} H_n\left(y, \frac{|u_n(x_n)|}{r_n}Dw_n\right)
\right|&\le C_1\frac{|u_n(x_n)|^{\gamma-1}}{r_n^{\gamma-2}}|Dw_n|^\gamma+
C_2\frac{r_n^2}{|u_n(x_n)|}\\
&\le C K^{\gamma-1} \left[|Dw_n|^\gamma+o(1)\right].
\end{split}
\end{equation*}
On the other hand, by the contradiction assumption and again Lemma \ref{lem:consDP}, 
\begin{equation}\label{eq:boundgn}
 \|g_n\|_{q;\Omega_n} = \frac{r_n^{2-\frac{N}{q}}}{|u_n(x_n)|} 
 \|f_n\|_{q;\Omega}\le \frac{r_n^{\alpha}}{|u_n(x_n)|} M \to 0.
\end{equation}
Finally, by Lemma \ref{lem:invasion}, $[w_n]_{\alpha;B_{2R}}\le 2$, hence Proposition \ref{prop:holdertoW2q} applies and the assertion follows.
\end{proof}
\begin{proof}[End of the proof of Thm. \ref{thm:main_intro}]
Let $R\ge 1$ and $n$ be sufficiently large. By Lemmas \ref{lem:iter} and 
\ref{lem:invasion} we know that $(w_n)_n$ is uniformly bounded in 
$W^{2,q}(B_R)$ and, by a diagonal procedure, $w_n\to w_\infty$ 
weakly in $W^{2,q}_\loc(\R^N)$. 

We want to pass to the limit in 
\eqref{eq:bu_sqn}. We already know that $g_n\to0$ in $L^q$, see equation 
\eqref{eq:boundgn}. Moreover, up to subsequences, we can assume $\bar x_n 
\to\bar x_\infty\in\overline{\Omega}$, whence \eqref{eq:assA} yields
\begin{equation}\label{eq:convA}
A_n(y) D^2w_n \to A_\infty D^2 w_\infty\quad\text{weakly in }L^q_\loc \ ,
\qquad\text{where }A_\infty=A(\bar x_\infty).
\end{equation}

Turning to the hamiltonian term, and taking into account \eqref{eq:assH}, 
we first notice that, since $\gamma q < q^*$, 
\[
h(\bar x_n + r_n y) \frac{|u_n(x_n)|^{\gamma-1}}{r_n^{\gamma-2}}|Dw_n|^\gamma
\to h_\infty |Dw_\infty|^\gamma \text{ strongly in }L^q_\loc,
\quad\text{where }0\le h_\infty \le K^{\gamma-1} h(\bar x_\infty)
\]
(recall the definition of $K$ in Lemma \ref{lem:consDP}). On the other hand
\eqref{eq:assH} implies
\[
\begin{split}
\frac{r_n^2}{|u_n(x_n)|} 
H_0&\left(\bar x_n + r_n y, \frac{|u_n(x_n)|}{r_n}Dw_n\right) \le 
C_1 \frac{|u_n(x_n)|^{\gamma_1-1}}{r_n^{\gamma_1-2}}|Dw_n|^{\gamma_1} + 
C_2 \frac{r_n^2}{|u_n(x_n)|}\\
& \le C_1\frac{r_n^{\gamma-\gamma_1}}{|u_n(x_n)|^{\gamma-\gamma_1}} K^{\gamma_1-1}|Dw_n|^{\gamma_1} + 
C_2 \frac{r_n^2}{|u_n(x_n)|}\to 0 
\end{split}
\]
in $L^q_\loc$, since $\gamma_1<\gamma$ (notice that, in \eqref{eq:assH}, 
we can assume w.l.o.g.\ $\gamma_1>1$). Summing up we obtain that
\begin{equation}\label{eq:convH}
 \frac{r_n^2}{|u_n(x_n)|} 
H_n\left(y, \frac{|u_n(x_n)|}{r_n}Dw_n\right) \to 
h_\infty |Dw_\infty|^\gamma\qquad \text{ strongly in }L^q_\loc.
\end{equation}
Then using \eqref{eq:bu_sqn} we obtain that the convergence in \eqref{eq:convA} is 
actually strong, and finally $w_\infty\in W^{2,q}_\loc(\R^N)$ solves
\[
-\tr \left(A_\infty D^2 w\right) + h_\infty|Dw|^\gamma = 0\qquad \text{in }\R^N.
\]
Now, in case $h_\infty>0$, Lemma \ref{lem:Lions} implies that $w$ is constant, in 
contradiction with Lemma \ref{lem:invasion}. On the other hand, in case 
$h_\infty=0$, $w_\infty$ is a harmonic function, globally H\"older continuous 
with exponent $\alpha<1$, which again implies the contradiction $w_\infty = \text{constant}$.
\end{proof}

We now prove Theorem \ref{thm:main_intro2}, which will be obtained as a straightforward consequence of Theorem \ref{thm:main_intro}.

\begin{proof}[Proof of Thm. \ref{thm:main_intro2}] Assume first that $q < N$, and fix any $\Omega''$ such that $\Omega' \ssubset \Omega'' \ssubset \Omega$. Then, Theorem \ref{thm:main_intro} yields the bound
\[
[u]_{\alpha, \Omega''} \le C, \qquad \alpha = 2 - \frac N q.
\]
The compact set $\overline{\Omega'}$ can be covered by a finite number of balls $B^k = B_r(x_k)$, such that $B_{2r}^k \subset \Omega''$. On any such ball, Proposition \ref{prop:holdertoW2q} applies (to $u - u(x_k)$), so that
\[
\|D^2 u\|_{q,B^k}\le K,
\]
and the conclusion follows. If $q \ge N$, one needs a few additional bootstrap steps. Pick any $N \frac \gamma{\gamma+1} < q' < N$, and bounds on $\|u\|_{W^{2,q'},\Omega''}$ are obtained as in the previous step. By Sobolev embeddings, these yield bounds on $\|u\|_{W^{1,(q')^*},\Omega''}$, and therefore $H(x, \nabla u)$ is bounded in $L^p(\Omega'')$, with $p = \frac{(q')^*}{\gamma}$, which is strictly bigger than $N$. One now applies Calder\'on-Zygmund regularity (see Lemma \ref{lem:GT}) to control $u$ in $W^{2,p}$, possibly on a smaller set. Since $p > N$, $u$ enjoys Lipschitz bounds, and again by Calder\'on-Zygmund one achieves the desired $W^{2,q}$ bounds.
\end{proof}


%

\medskip
\begin{flushright}
\noindent Marco Cirant\\
Dipartimento di Matematica ``Tullio Levi-Civita'', Universit\`a di Padova\\
Via Trieste 63, 35121 Padova, Italy\\
\verb"cirant@math.unipd.it"\\
\medskip
\noindent Gianmaria Verzini\\
Dipartimento di Matematica, Politecnico di Milano\\
piazza Leonardo da Vinci 32, 20133 Milano, Italy\\
\verb"gianmaria.verzini@polimi.it"
\end{flushright}

\end{document}